\documentclass{svjour3}
\usepackage{amsfonts}

\newcommand{\GF}{\mathop{\mathrm{GF}}}
\newcommand{\PG}{\mathop{\mathrm{PG}}}

\begin{document}
\title{Diagonal groups and arcs over groups}
\dedication{To Aart Blokhuis}
\author{R. A. Bailey \and Peter J. Cameron \and
Michael Kinyon$^*$\thanks{$^*$ Partially supported by Simons Foundation Collaboration Grant 359872 and by Fundaç\~{a}o para a Ci\^{e}ncia e a Tecnologia (Portuguese Foundation for Science and Technology) grant PTDC/MAT-PUR/31174/2017.} \and 
Cheryl E. Praeger$^\dag$\thanks{$^\dag$ Supported by Australian Research Council Discovery Grant DP160102323}}
\institute{R. A. Bailey \and Peter J. Cameron \at School of Mathematics and
Statistics, University of St Andrews, St Andrews, Fife KY16 9SS, UK
\email{rab24@st-andrews.ac.uk, pjc20@st-andrews.ac.uk} \and
Michael Kinyon \at Department of Mathematics, University of Denver, Denver
CO 80208, USA
\email{mkinyon@du.edu} \and
Cheryl E. Praeger \at Department of Mathematics and Statistics,
University of Western Australia, Perth WA 6009, Australia
\email{cheryl.praeger@uwa.edu.au}}
\maketitle

\begin{abstract}
In an earlier paper by three of the present authors and Csaba Schneider,
it was shown that, for $m\ge2$, a set of $m+1$ partitions of a set $\Omega$,
any $m$ of which are the minimal non-trivial elements of a Cartesian lattice,
either form a Latin square (if $m=2$), or generate a join-semilattice of
dimension $m$ associated with a diagonal group over a base group $G$.

In this paper we investigate what happens if we have $m+r$ partitions with
$r\geq 2$, any $m$ of which are minimal elements of a Cartesian lattice. If
$m=2$, this is just a set of mutually orthogonal Latin squares. We consider
the case where all these squares are isotopic to Cayley tables of groups, and
give an example to show the groups need not be all isomorphic. For $m>2$,
things are more restricted. Any $m+1$ of the partitions generate a
join-semilattice admitting a diagonal group over a group $G$. It may be that
the groups are all isomorphic, though we cannot prove this. Under an extra
hypothesis, we show that $G$ must be abelian and must have three
fixed-point-free automorphisms whose product is the identity. (We describe
explicitly all abelian groups having such automorphisms.) Under this
hypothesis, the structure gives an orthogonal array, and conversely in some
cases.

If the group is cyclic of prime order $p$, then the structure corresponds
exactly to an arc of cardinality $m+r$ in the $(m-1)$-dimensional projective
space over the field with $p$ elements, so all known results about
arcs are applicable. More generally, arcs over a finite field of order $q$
give examples where $G$ is the elementary abelian group of order $q$. These
examples can be lifted to non-elementary abelian groups using $p$-adic
techniques.

\end{abstract}

\section{Introduction}

The origin of this paper was a realisation that, as sets of mutually orthogonal
Latin squares extend the notion of Latin squares to more objects, and arcs in
finite projective spaces extend to higher dimension, there should be a way
to define and study objects realising both of these extensions. Given the
fundamental work of Aart Blokhuis in finite geometry, especially
on arcs~\cite{bbt}, we regard this as a fitting tribute to him.

Central to our work is the notion of \emph{diagonal groups} and the
structures they act on.
Diagonal groups $D(G,m)$ are one of the families of primitive permutation
groups arising in the celebrated O'Nan--Scott theorem. In this theorem, the
group $G$ is a finite simple (or characteristically simple) group.
In~\cite{bcps}, the authors considered diagonal groups with an
arbitrary group $G$ (not necessarily finite), and defined a geometric object
having the diagonal group as its automorphism group. This object was called
a \emph{diagonal semilattice}. We now outline the details.

We work within the lattice $\mathcal{P}(\Omega)$ of partitions of a
set~$\Omega$.
A \emph{Cartesian lattice} of dimension $m$ over an alphabet $A$ is defined
as follows: $\Omega=A^m$, and for any subset $J$ of $\{1,\ldots,m\}$, we
define the partition $Q_J$ of $\Omega$ in which two $m$-tuples
$(a_1,\ldots,a_m)$ and $(b_1,\ldots,b_m)$ belong to the same part if and only
if $a_j=b_j$ for all $j\notin J$. These partitions form a lattice isomorphic to
the Boolean lattice on $\{1,\ldots,m\}$ (the lattice of all subsets of
$\{1,\ldots,m\}$); the map $J\mapsto Q_J$ is an isomorphism from the
Boolean lattice to the Cartesian lattice.

\begin{definition}
The \emph{diagonal group} $D(G,m)$ can be defined as a group of permutations
of the set $\Omega=G^m$ generated by the following permutations:
\begin{itemize}
\item right translations by elements of $G^m$;
\item left translations by elements of the subgroup
$\delta(G)=\{(g,g,\ldots,g):g\in G\}$ of $G^m$;
\item automorphisms of $G$ (acting in the same way on all coordinates);
\item permutations of the coordinates;
\item the map
\[(g_1,g_2,\ldots,g_m)\mapsto(g_1^{-1},g_1^{-1}g_2,\ldots,g_1^{-1}g_m).\]
\end{itemize}
\end{definition}

For $i=1,\ldots,m$, let $G_i$ be the $i$th coordinate subgroup of $G^m$,
the set of $m$-tuples $(g_1,\ldots,g_m)$ with $g_j=1$ for $j\ne i$; and
let $G_{m+1}$ be the subgroup $\delta(G)$. Then let $Q_i$ be the partition of
$G^m$ into right cosets of $G_i$, for $i=1,\ldots,m+1$. 

The main results of \cite{bcps} can be stated as follows:

\begin{theorem}
\label{th:magic}
  The join-semilattice $\Lambda$ generated by the partitions $Q_1,\ldots,Q_{m+1}$
has the properties
\begin{itemize}
\item any $m$ of $\{Q_1,\ldots,Q_{m+1}\}$ generate a Cartesian lattice under
  join;
\item the automorphism group of $\Lambda$ is the diagonal group $D(G,m)$.
\end{itemize}
\end{theorem}

We call $\Lambda$ a \emph{diagonal semilattice}, and denote it by 
$\mathcal{D}(G,m)$.

\begin{theorem}\label{t:main}
Let $Q_1,\ldots,Q_{m+1}$ be partitions of $\Omega$, where $m\geq2$. Suppose that
any $m$ of these partitions generate an $m$-dimensional Cartesian lattice, in
which they are the minimal non-trivial elements. Then one of the following
holds:
\begin{itemize}
\item $m=2$ and there is a Latin square $L$, unique up to paratopism, such that
$\Omega$ is the set of cells of $L$, and $Q_1,Q_2,Q_3$ are the partitions of
$\Omega$ corresponding to the rows, columns and letters of $L$;
\item $m\geq 3$ and there is a group $G$, unique up to isomorphism, such that
$Q_1,\ldots,Q_{m+1}$ generate the diagonal semilattice $\mathcal{D}(G,m)$.
\end{itemize}
\end{theorem}

A \emph{paratopism} between two Latin squares is most easily defined here as a
bijection between the set of cells of the first and that of the second which
carries the three partitions (letters, rows and columns) of the first set to
the three partitions of the second set in some order. If rows map to rows,
columns to columns and letters to letters, the map is called an
\emph{isotopism}.

In this paper, we consider what happens when we have $m+r$ partitions
satisfying the hypotheses of this theorem with larger values of $r$. We will 
show that
\begin{itemize}
\item If $m=2$, then the partitions form the rows, columns, and letters in
$r$ mutually orthogonal Latin squares. The case where all the Latin squares
are isotopic to Cayley tables of groups is particularly interesting, and we
give an example with $r=2$ where the four groups fall into three different
isomorphism classes.
\item If $m\ge3$ and $r\ge2$, then under an additional assumption (which we
call \emph{regularity}) the groups $G$ obtained by applying
Theorem~\ref{t:main} to any $(m+1)$-tuple of partitions are
all isomorphic, are abelian, and this unique abelian group admits three
fixed-point-free automorphisms whose product is the identity. We describe all
abelian groups having such automorphisms, and give examples based on $p$-adic
lifting of arcs in finite projective spaces. We also describe the relation
of our work to orthogonal arrays.
\end{itemize}

We introduce some notation. Let $t(m,n)$ be the greatest value of $r$ for 
which such a set of partitions of a set of cardinality $n^m$ exists. (We
assume that $m\ge2$ and $n\ge2$.) For $m=2$, this is the maximum number of
mutually orthogonal Latin squares of order $n$ (usually denoted by $N(n)$ in
the literature). Further, when $m=2$ we denote by $t_g(2,n)$ the maximum in
the case where all the Latin squares obtained by taking the partitions three
at a time are Cayley tables of groups.
(We do not need to define this for $m\geq 3$, because Theorem~\ref{t:main}
shows that, in this case, any set of $m+1$ of the partitions defines a group.)
For any given group~$G$, we also
denote by $T(m,G)$ the maximum number $r$ for which there are $m+r$ partitions
satisfying our hypothesis such that any $m+1$ of them define a group
isomorphic to $G$.  Thus $T(m,G) \leq t(m, \left| G \right|)$. 

Part of our purpose here is to consider these functions and give some upper and
lower bounds. We will see that our problem involves several other parts of
combinatorics and finite geometry, including mutually orthogonal Latin
squares, the Hall--Paige conjecture, and arcs in finite projective spaces.

\section{The case $m=2$}
\label{sec:2}

Suppose that we have a collection of $r+2$ partitions of $\Omega$ with the
property that any two of them give $\Omega$ the structure of an $n\times n$
grid. Any further partition can be represented by a set of letters 
corresponding to the parts of the partition, and the hypothesis implies that
the letters constitute a Latin square of order $n$ on the square array.
Further, any two of the resulting Latin squares are orthogonal. So we have
precisely a set of $r$ mutually orthogonal Latin squares (MOLS) of order~$n$.

Note that the maximum number of orthogonal Latin squares of order $n$ satisfies
$t(2,n)\le n-1$, with equality if $n$ is a prime power: see \cite[p.~158]{dk}.

A set of $r$ MOLS defines $r+2\choose 3$ Latin squares, since any triple of
the partitions gives such a square. We will say that we have a set of
\emph{mutually orthogonal group squares} (MOGS) if all of these Latin squares
are isotopic to Cayley tables of groups.

We note that there is a test, the \emph{quadrangle criterion}, to determine
whether a Latin square is isotopic to a Cayley table of a group, due to
Frolov~\cite{frolov} (see~\cite[Theorem 1.2.1]{dk} and the following text for
discussion); and a theorem of Albert~\cite[Theorem~2]{albert} shows
that, if so, then the group is unique up to isomorphism.

MOLS have been studied since Euler, and we have nothing to add in general.
But note that the classical set of $q-1$ MOLS of order $q$ (for prime powers
$q$) associated with the Desarguesian projective plane of order $q$ does indeed
form a set of MOGS, where all the groups are isomorphic to the additive group
of the finite field of order $q$. So $t_g(2,n)\le n-1$, with equality if $n$ is
a prime power; and $T(2,G)=q-1$ if $q$ is a prime power and $G$ is elementary
abelian of order $q$.

More interesting to us is a remarkable example of two MOLS of order $8$
where all of the Latin squares are Cayley tables of groups, but the groups
are not all isomorphic:

\[\begin{array}{|c|c|c|c|c|c|c|c|}
\hline
11 & 22 & 33 & 44 & 55 & 66 & 77 & 88\\
\hline
42 & 34 & 21 & 13 & 86 & 78 & 65 & 57\\
\hline
53 & 61 & 74 & 82 & 17 & 25 & 38 & 46\\
\hline
84 & 73 & 62 & 51 & 48 & 37 & 26 & 15\\
\hline
35 & 47 & 16 & 28 & 71 & 83 & 52 & 64\\
\hline
76 & 85 & 58 & 67 & 32 & 41 & 14 & 23\\
\hline
27 & 18 & 45 & 36 & 63 & 54 & 81 & 72\\
\hline
68 & 56 & 87 & 75 & 24 & 12 & 43 & 31\\
\hline
\end{array}\]

The four groups are as follows. Here $G_i$ denotes the group obtained by
omitting
the $i$th of the four partitions (rows, columns, first letter, second letter);
so $G_4$ and $G_3$ denote the groups whose multiplication tables are given by
the first and second letters in the array.
\begin{description}
\item[$G_4$:] $C_2\times C_2\times C_2$
\item[$G_3$:] $D_8$
\item[$G_2$:] $C_2\times C_4$
\item[$G_1$:] $D_8$
\end{description}

The proof of the Hall--Paige conjecture~\cite{hp} by Wilcox, Evans and 
Bray~\cite{wilcox,evans,bccsz} shows that the Cayley table of a group $G$
has an orthogonal mate if and only if the Sylow $2$-subgroups of $G$ are
trivial or non-cyclic. In particular, no group of order congruent to $2$
(mod~$4$) satisfies this condition, so in the earlier language we have the
second part of the following proposition.

\begin{proposition}
\begin{itemize}
\item If $q$ is a prime power and $G$ is an elementary abelian group of order
$q$, then $T(2,G)=q-1$; in particular, $t_g(2,q)=q-1$.
\item If $n\equiv2\!\!\!\pmod{4}$, then $t_g(2,n)=1$.
\end{itemize}
\end{proposition}

\begin{problem}
For $r>1$, is there a set of $r$ MOGS such that all $r+2\choose 3$ groups are
pairwise non-isomorphic?
\end{problem}

\begin{problem}
Given a group $G$, what is the largest $r$ such that there exists a set
of $r$ MOGS for which all $r+2\choose 3$ groups are isomorphic to $G$? That is,
what is $T(2,G)$?
\end{problem}

Owens and Preece \cite{OP1,OP2} investigated the number of different species
of Latin squares that occur in the seven different affine planes of order~$9$. 
(A \emph{species} is an equivalence class under paratopism.)
Two of the affine planes have just one species of Latin square, which is the
Cayley table of $C_3 \times C_3$ in both cases.
A third affine plane has some choices of which two partitions define rows
and columns for which all the Latin squares are Cayley tables of
$C_3 \times C_3$.
The
Cayley table of $C_9$ never occurs.  Egan and Wanless \cite{EW} repeated this
investigation, and extended it to other sets of MOLS.  For a set of three
MOLS,  there are ten ways of choosing three of the five partitions to form a
single  Latin square. For MOLS of order~$9$, Egan and Wanless found
that the number of different species occuring can be any integer in
$\{1, \ldots, 10\}$.

Another result bearing on this question can be found in the paper of
Franceti\'c, Herke and Wanless~\cite{fhw}. They define the notion of the
\emph{parity} of a Latin square, and prove (among other things) that, if
$n\equiv2$ (mod~$4$), then there is no complete set of $n-1$ Latin squares
of order $n$ in which all ${n+1\choose3}$ Latin squares are isotopic. Thus,
if we extend our notation $T(2,G)$ to quasigroups, so that $T(2,Q)$ is the
maximum number of MOLS in which all definable Latin squares are paratopic to
the Cayley table of the quasigroup $Q$, then for $|Q|\equiv2$ (mod~$4$) we have
$T(2,Q)<|Q|-1$.

\section{The case $m>2$}

As noted in \cite{bcps}, there are several definitions of ``Latin cube'' in
the literature; the one relevant to the proof of the Main Theorem in that
paper is one of these, and not the most popular. The situation for
orthogonal Latin cubes is if anything worse,
see~\cite{ds,mw,mullen,potapov,trenkler}. To avoid causing
more confusion, we will use the name \emph{diagonal semilattices} for the
objects appearing in \cite{bcps}.
So the objects to be studied here are sets of
\emph{mutually orthogonal diagonal semilattices}, or MODS for short.

Thus, a set of $r$ MODS of \emph{dimension} $m$ and \emph{order} $n$ is a
collection of $m+r$ partitions $Q_1,\ldots,Q_{m+r}$ of a set $\Omega$ of
cardinality $n^m$, with the property that any $m$ of these partitions are
the minimal non-trivial elements in an $m$-dimensional Cartesian lattice
on $\Omega$. According to \cite{bcps}, if $r=1$ (and $m>2$) then there is a
group $G$ of order $n$, unique up to group isomorphism,
such that $\Omega$ can be identified with $G^m$, and the
partitions $Q_i$ are the coset partitions of $\Omega$ with respect to subgroups
$G_1,\ldots,G_{m+1}$, where $G_i$ acts by right multiplication on the $i$th
coordinate of elements of $G^m$, fixing the entries in all other coordinates,
for $i=1,\ldots,m$, and $G_{m+1}$ acts by left multiplication of all entries by 
the same group element:
\[x\colon (g_1,\ldots,g_m)\mapsto(x^{-1}g_1,\ldots,x^{-1}g_m).\]
(The $x^{-1}$ is to ensure that the requirements for a (right) action are
satisfied.)

Let us say that a set of $r$ MODS is \emph{regular} if all the partitions
are right coset partitions of subgroups of order $n$ in $G^m$. The main
problem, which we have not been able to solve, is:

\begin{problem}
Does there exist a non-regular set of MODS (with $m>2$ and $r>1$)?
\end{problem}

\begin{proposition}\label{p:mods}
In a regular set of MODS, every $(m+1)$-tuple of partitions gives rise to
a diagonal semilattice over a group $G$ which is independent of the tuple
of partitions chosen. Moreover, $G$ is an abelian group which admits three
fixed-point-free automorphisms whose product is the identity.
\end{proposition}

\begin{proof}
It suffices to prove this in the case $m=3$, $r=2$, which we assume from 
now on.

Each partition $Q_i$ is the coset partition corresponding to a subgroup $G_i$
of $G^3$, where we can assume that $G_1,G_2,G_3$ are the coordinate
subgroups, as defined above. This implies that $G_1$, $G_2$ and $G_3$ pairwise
commute elementwise. Since the choice of these three subgroups was arbitrary,
we see that $G_i$ and $G_j$ commute elementwise for any choice of $i,j$.

Any further partition must be the coset 
partition of a subgroup intersecting the product of fewer than $m$ of these
subgroups in the identity. Such a subgroup must be a ``diagonal'' of the
form $\{(g,g^\alpha,g^\beta):g\in G\}$, where $\alpha$ and $\beta$ are
automorphisms of $G$. Moreover we may take $G_4$ to be the usual diagonal
subgroup, defined by the choices $\alpha=\beta=1$.

Now consider $G_5$, and write it in the form
\[G_5=\{(g,g^\alpha,g^\beta):g\in G\}\]
for some automorphisms $\alpha$ and $\beta$.

Now $G_5$ must commute with $G_4$ elementwise.
But the projection onto the first coordinate
induces an isomorphism on both $G_4$ and $G_5$, with image $G$ in both cases;
so $G$ is abelian.

By definition, any three of $G_1$, $G_2$, $G_3$, $G_4$ and $G_5$ generate
their direct product $G^3$. Now consider $G_3G_4G_5$. Since 
\[G_3G_4=\{(g,g,h):g,h\in G\},\]
we see that the only solution of $g^\alpha=g$ must be $g=1$; in other words,
$\alpha$ is a fixed-point-free automorphism.
Replacing $G_3$ by $G_2$ and $G_1$ in turn, the same argument
shows that $\beta$ and $\alpha^{-1}\beta$ are also fixed-point-free
automorphsms. Putting $\gamma=\beta^{-1}\alpha$, we see that $\alpha$,
$\beta^{-1}$ and $\gamma$ are fixed-point-free automorphisms whose product is
the identity, as required.
\end{proof}

It is possible to describe the abelian groups which have such triples of
automorphisms:

\begin{proposition}
The following are equivalent for finite abelian groups $G$:
\begin{enumerate}
\item $G$ admits three fixed-point-free automorphisms whose product is the
identity;
\item if $G$ is written as a direct product of cyclic groups of prime power
orders, then factors whose order is a power of $2$ or of $3$ occur with
multiplicity greater than~$1$.
\end{enumerate}
\end{proposition}

\begin{proof}
Let $G$ be such a group. We will use the principle that, if
$H$ and $K$ are characteristic subgroups of $G$ with $H\leq K$, then
$K/H$ admits three fixed-point-free automorphisms whose product is the
identity also. So if $G$ has this property, so do its Sylow subgroups (its
primary components); so we may assume that $G$ is a $p$-group.

Suppose that $p=2$ or $p=3$, and that in the expression for $G$ as a direct
product of cyclic groups, some cyclic group (say $C_{p^e}$) occurs with
multiplicity~$1$. Taking the quotient by $\{g\in G:g^{p^{e-1}}=1\}$, we may
assume that $e=1$. Now let $K$ be the set of elements of order dividing $p$,
and $H=G^p$; then $K/H\cap K$ is cyclic of order $p$. But it is easy to see that
cyclic groups of orders $2$ and $3$ do not have triples of automorphisms as
required.

In the other direction, cyclic groups of $p$-power order with $p\geq5$,
and groups $(C_{p^e})^d$, for $p>3$ and $d>1$, do admit such triples.
\end{proof}

In particular, for any $n$ not congruent to $2$ mod~$4$ or to $\pm3$ mod~$9$,
there is an abelian group of order $n$ with this property.
One example is the direct product of elementary abelian groups, whose
exponent is square-free; the condition on~$n$ ensures that, for $p=2$
and $p=3$, the Sylow $p$-subgroup is either trivial or non-cyclic.

\medskip

The examples to be described in the following sections are all regular in the
sense defined in this section.

\section{Orthogonal arrays}

If $G$ is an abelian group then its \textit{dual group} $G^*$ consists of the
irreducible complex characters of $G$, and is isomorphic to $G$.  These are
frequently used by statisticians in factorial design.  For example, if
$G=C_p^4$ for some prime $p$, then typically $G$ is written as
$\langle a\rangle  \times \langle b \rangle \times \langle c \rangle
\times \langle d\rangle$ and $G^*$ as $\langle A \rangle \times \langle B
\rangle \times \langle C \rangle \times \langle D\rangle$, where $A$ simply
picks out the power of $a$ and raises $\exp(2\pi\mathrm{i}/p)$ to that power.

The elements of $G^*$ can be thought of as partitions of $G$.
(Strictly speaking, the character $A$ defines the partition of $G$ whose parts
are the inverse images of each complex number in the image of $A$.)
If $G$ is the direct product of $m$ abelian groups of order~$n$, then a set of
$m+r$ such partitions of $G$ is called an \textit{orthogonal array of
strength~$m$} and \textit{index~$1$} if any $m$ of them  form the maximal
elements in a Cartesian lattice on $G$. 

More generally, an orthogonal array with $k$ factors having strength $m$ and
index $\lambda$ over an alphabet $A$ of size~$n$
is a set of $k$-tuples of elements of $A$
with the property that, given any $m$ distinct coordinates $i_1,\ldots,i_m$ and
any $m$ arbitrary elements $a_1,\ldots,a_m$ of $A$, there are exactly $\lambda$
tuples having $a_j$ in position $i_j$ for $j=1,\ldots,m$.
The numbers $n$ and $k$ are sometimes called the \emph{number of levels}
and \emph{number of factors} respectively.

We are only concerned
with index $1$. In this case, each coordinate defines a partition of the set
$\Omega$ of $k$-tuples according to the letter in that coordinate, and this
set of $k$ partitions has the property that any $m$ of them are the
\emph{maximal} elements in a Cartesian lattice of dimension $m$.

For example, if $G=C_p^4$ with $p\geq 5$ then
$\{A,B,C,D,ABCD,AB^2C^3D^4\}$ is an orthogonal array of strength $4$.
This is usually denoted $\mathrm{OA}(p^m,6,p,4)$: see \cite{HSS}.

In fact, the complete set of MOLS of order~$9$ given by Fisher and Yates in
\cite{FY} was constructed in this way, using $C_3^4$ as the underlying set.
The rows are labelled by pairs of values of $A$ and $B$, while the columns
are labelled by pairs of values of $C$ and $D$.  The letters in the first
square are identified by pairs of values of $AC^2$ and $BD^2$; and so on.  It
is thus no surprise that all eight Latin squares are Cayley tables of
$C_3 \times C_3$.  What was surprising to its authors was that this set of
MOLS is not isomorphic to the one given in \cite{FDE}.  They originally
thought that it was, but Fisher apologised for the mistake in \cite{Fapol}.
In fact, these are the first two affine planes discussed at the end of
Section~\ref{sec:2}.

Let us return to orthogonal arrays.
The concept of orthogonal array is the dual notion
(in the sense of reversing the partial order of refinement of partitions)
of the property stated in the first part of Theorem~\ref{th:magic}.
If the orthogonal array is defined by an abelian group, then taking the
dual group 
also reverses the order of refinement.
Hence the dual of each such orthogonal array gives a set of MODS.

In the running example, the dual of the orthogonal array is the set of
subgroups $\langle a \rangle$, $\langle b \rangle$, $\langle c \rangle$,
$\langle d \rangle$, $\langle abcd \rangle$, $\langle ab^2c^3d^4 \rangle$.
These have the property that every subset of four of them generate their direct
product:  in other words, their coset partitions form the minimal non-trivial
partitions in a join semi-lattice.

If we write this in more standard notation over the field $\mathrm{GF}(p)$,
then we have six vectors $(1,0,0,0)$, $(0,1,0,0)$, $(0,0,1,0)$, $(0,0,0,1)$,
$(1,1,1,1)$ and $(1,2,3,4)$ with the property that every four are
linearly independent.  Multiplying any of these vectors by a non-zero
scalar does not affect this, so we are effectively in projective geometry, and
we have six points, any four of which are in general position.

We conclude this section by showing that a regular set of MODS with $r>1$
and $m>2$ does indeed give rise to an orthogonal array.

\begin{proposition}
  Given a regular set of MODS of dimension $m$ and order~$n$
  with $m+r$ partitions, for
$r\ge2$, we can construct from it an orthogonal array of strength $m$ and
index $1$ with $m+r$ factors, each with $n$ levels.
\end{proposition}

\begin{proof}
According to Proposition~\ref{p:mods}, a regular set of $m+r$ MODS of order $n$
and dimension $n$, with $m>2$ and $r>1$, is realised by the coset partitions
of $G^m$ by $m+r$ subgroups isomorphic to $G$, where $G$ is an abelian group
of order $n$ (and has three fixed-point-free automorphisms whose product is
the identity).

Now the dual group $(G^m)^*$ defines an orthogonal array of strength $m$ and
index~$1$ over the alphabet $A=G^*$ by the following rule. First identify
each subgroup $G_i$ with $G$ by a fixed isomorphism $\psi_i$. Then take
$\phi\in(G^m)^*$; map $\phi$ to the $n$-tuple $a(\phi)$, where
$a(\phi)_i$ is the restriction of $\phi$ to $G_i$.
For $i=1$, \ldots, $m$, use the isomorphism $\psi_i$ to identify
$a(\phi)_i$ with an element of $G^*$.

Given any set $i_1,\ldots,i_m$ of $m$ distinct indices,
$G$ is the direct sum of the
groups $G_{i_1}, \ldots, G_{i_m}$, and so an element of $G^*$ is uniquely
defined by its restriction to these subgroups; conversely, any choice of
elements of $G_{i_j}^*$ for $j=1,\ldots,m$ defines a unique homomorphism of 
$G$. So we have an orthogonal array, as claimed.
\end{proof}

\section{Frobenius groups}

A \emph{Frobenius group} is a finite group $G$ with a non-trivial proper
subgroup $H$ (called the \emph{Frobenius complement}) such that $H\cap H^g=1$
for all $g\in G\setminus H$, where $H^g$ is the conjugate $g^{-1}Hg$. The
theorem of Frobenius shows that the identity together with elements lying in
no conjugate of $H$ form a normal subgroup $N$, the \emph{Frobenius kernel}.
The celebrated theorem of Thompson asserts that the Frobenius kernel is
nilpotent. 

Alternatively, a Frobenius group is a transitive permutation group $G$ in which
the one-point stabilisers are non-trivial but all two-point stabilisers are
trivial. The one-point stabilisers are the Frobenius complements, and the
Frobenius kernel is a regular normal subgroup.

We refer to Passman~\cite{passman} for an account of this material.

\begin{theorem}
Let $G$ be a Frobenius group whose Frobenius kernel $N$ is abelian, with
Frobenius complement $H$. Then there is a set of $|H|$ MOGS of order $|N|=n$
such that each of the $|H|+2\choose 3$ Latin squares is isotopic to the Cayley
table of $N$.
\end{theorem}

\begin{proof}
Each square has rows and columns indexed by $N$; the squares are indexed by $H$.
The square $L_h$ has $(x,y)$ entry $L_h(x,y)=xy^h$. (This is analogous to the
usual finite field construction of MOLS.) Now $L_h$ is isotopic to the
Cayley table of $N$, since if we relabel the column previously labelled $y^h$
as $y$ then we recover the Cayley table of $N$ (which indeed is $L_1$).

To show orthogonality, take distinct $h,k\in H$ and $a,b\in N$; we need to
show that the equations $xy^h=a$ and $xy^k=b$ have a unique solution
$(x,y)\in N\times N$.
But these equations imply $y^{-1}kh^{-1}y=kb^{-1}ah^{-1}$, so $y$ conjugates
$kh^{-1}$ to $kb^{-1}ah^{-1}$. But the centraliser in $H$ of a non-identity
element of $N$ is trivial. So if $y_1$ and $y_2$ were two such elements, then
$y_1y_2^{-1}$ would commute with $kh^{-1}$, so $y_1=y_2$. Thus $y$, and hence
also $x$, is uniquely determined. (We have proved that there cannot be more
than one solution: now counting shows there is exactly one.)

Now we have to show that, of the $|H|+2$ partitions corresponding to rows,
columns, and the $|H|$ squares, if we choose any two to be new rows and
columns, it is still true that all the squares are isotopic to the Cayley
table of $N$. Recall that two squares are isotopic if there are permutations
of the rows, columns and letters which transform one to the other. So we
need to show that, in each case, there are bijections $\phi,\chi,\psi$ of $N$
such that the entry in row $u$ and column $v$ of the second square is given by
$\psi(\phi(u)\chi(v))$, where inside the brackets we have the group operation
in $N$. Different squares will of course require different choices of
$\phi,\chi,\psi$. We saw an example in the first paragraph, of this proof,
where $\phi$ and $\psi$ are the identity and $\chi(v)=v^h$.

We begin with a couple of observations.

\paragraph{Note 1:}
For any $k\in H$, there is a symmetry which maps $L_h$ to $L_{kh}$ and
conjugates the column labels by $k$; and there is a symmetry which swaps
rows and columns, and replaces $L_h$ by $L_{h^{-1}}$ with its letters
conjugated by $h$ for each $h\in H$.

For the first, $L_{kh}(x,y)=xy^{kh}=x(y^k)^h$. For the second, 
$(yx^{h^{-1}})^h=y^hx=xy^h$.

\paragraph{Note 2:}
For $h\in H$, $h\ne1$, the map $\zeta_h\colon N\to N$ given by
$\zeta_h(x)=x^{-1}x^h$
is a bijection; since $N$ is abelian, it is an automorphism.
Thus we may define $\eta_h\colon N\to N$ to be the inverse of $\zeta_h$.

For suppose that $x^{-1}x^h=y^{-1}y^h$. Then $yx^{-1}=(yx^{-1})^h$. Since
conjugation by $h$ is a fixed-point-free automorphism of $N$, this gives
$yx^{-1}=1$, so $x=y$.

Since $N$ is abelian, 
\begin{eqnarray*}
\zeta_h(xy) &=& y^{-1}x^{-1} x^hy^h,\\
\zeta_h(x)\zeta_h(y) &=& x^{-1}x^hy^{-1}y^h,
\end{eqnarray*}
and the right-hand sides are equal.

\medskip

Now we have to deal with the cases where the two partitions defining the
rows and columns of the square are no longer the original ones.
We have seen in Note~1 that swapping rows and columns gives a symmetry.
Therefore, if we use one of rows and columns, we can assume that it is rows.
Note~1 also shows that if we use the partition corresponding to an element
$h\in H$, we can assume that $h=1$. So there are two cases.

\paragraph{Case 1:} We use the row partition as rows and the partition
corresponding to $h=1$ as columns. Thus, the row and column labels are
$x$ and $xy=z$.

Consider the square corresponding to the former columns, with $(x,z)$ entry
$y$. Since $y=x^{-1}z$, this square is isotopic to the Cayley table of $N$.

Now consider the square $L_h$, with $(x,z)$ entry $xy^h$. Now
\[xy^h=x.x^{-h}z^h=(x^{-1})^{-1}(x^{-1})^hz^h=\zeta_h(x^{-1})z^h.\]
Since inversion, $\zeta_h$, and conjugation by $h$ are bijections, this is an
isotope of the Cayley table of $N$. (Take $\phi(x)=\zeta_h(x^{-1})$,
$\chi(z)=z^h$ and $\psi$ the identity map.)

\paragraph{Case 2:} We use the partition corresponding to the identity as
rows, the partition corresponding to $h$ as columns, and the partition
corresponding to $k$ as letters. Thus, if the corresponding square has
$(u,v)$ entry $w$, then $u=xy$, $v=xy^h$, and $w=xy^k$.

Solving the first two equations for $x$ and $y$ gives
\begin{eqnarray*}
u^{-1}v=y^{-1}y^h=\zeta_h(y),&\hbox{so}&y=\eta_h(u^{-1}v),\\
v^{h^{-1}}u^{-1}=x^{h^{-1}}x^{-1}=\zeta_{h^{-1}}(x),&\hbox{so}&
x=\eta_{h^{-1}}(v^{h^{-1}}u^{-1}).
\end{eqnarray*}
Thus
\[w=xy^k=\eta_{h^{-1}}(v^{h^{-1}}u^{-1})(\eta_h(u^{-1}v))^k.\]
Since $G$ is abelian and $\eta_h$, $\eta_{h^{-1}}$, inversion and conjugation
are isomorphisms, we can write this in the form $w=\phi(u)\chi(v)$, where
\[\phi(u)=\eta_{h^{-1}}(u^{-1})(\eta_h(u^{-1}))^k,\qquad
\chi(v)=\eta_{h^{-1}}(v^{h^{-1}})(\eta_h(v))^k.\]
Taking these functions $\phi$ and $\chi$ and the identity for $\psi$ gives the
required isotopism.
\end{proof}

\section{Higher-dimensional examples}

Let $q$ be a prime power, and let $G$ be the additive group of the finite
field $\GF(q)$ of order $q$. An \emph{arc} in the projective space
$\PG(m-1,q)$ is a set of points, any $m$ of which span the space. It is
called a $k$-arc if its cardinality is $k$.

In  vector space terms, it is a set of $1$-dimensional subspaces of the
$m$-dimensional vector space over $\GF(q)$, such that spanning vectors of
any $m$ of the spaces form a basis for the vector space.

Now it is clear that the coset partitions of any $m$ of these $1$-dimensional
subspaces are the minimal elements of a Cartesian lattice. Thus we have:

\begin{proposition}
If there exists an $(m+r)$-arc in $\PG(m-1,q)$, then $T(m,G)\ge r$.
\end{proposition}

The maximum cardinality of arcs in finite projective space was first
studied by Segre in the 1950s. In~\cite{segre}, he raised some fundamental
questions which have directed research since. A milestone in their study
was the paper of Blokhuis, Bruen and Thas~\cite{bbt}. We refer to the
recent survey by Ball and Lavrouw~\cite{bl} for further information.

The simplest example is the \emph{normal rational curve}. Let
$a_1,a_2,\ldots,a_q$ be the elements of $\GF(q)$. For $m\le q+1$, consider
the vectors $(1,a_i,a_i^2,\ldots,a_i^{m-1})$ for $i=1,\ldots,q)$ together
with $(0,0,\ldots,1)$. Any $m$ of these vectors form a basis for $\GF(q)^m$.
For if the last vector is not included, then the vectors are the rows of a
Vandermonde matrix, whose determinant is non-zero; the argument is similar
if the last vector is included.

We now present examples in other abelian groups, specifically homocyclic
$p$-groups. Such a group $G$ is a direct power of a cyclic group of prime
power order, say $G=(C_{p^e})^d$. Arcs in projective spaces give examples
with $e=1$, as we have seen. The construction involves lifting to a $p$-adic
number field and taking quotients; all necessary information can be found in
Henri Cohen's book~\cite{cohen}.

Let $q=p^d$. The splitting field of the polynomial $X^q-X$ over the field
$\mathbb{Q}_p$ of $p$-adic numbers is an extension $F$ of $\mathbb{Q}_p$ of
degree $d$. Its integers form a local ring $R$, with maximal ideal $M$
satisfying $R/M\cong\mathbb{F}_q$.

Let $S$ be the set of roots of $X^q-X$. The non-zero elements of $S$ form a
cyclic group of order $q-1$.

For a positive integer $m\le q+1$, the set
\[\{(1,u,u^2,\ldots,u^{m-1}):u\in S\}\cup\{(0,0,\ldots,0,1)\}\]
of vectors in $F^m$ has the property that any $m$ of its elements form a 
basis for $F^m$. The argument is the same as in the finite field case.

Reducing this set of vectors modulo the ideal $M$ gives a set of $q+1$
vectors in $(\mathbb{F}_q)^m$, any $m$ forming a basis for this space. The
$1$-dimensional subspaces they span form the standard representation of the
normal rational curve in $\mathrm{PG}(m-1,q)$.

Now fix an integer $e\ge2$. If we reduce modulo $M^e$, we obtain $q+1$
elements in the group $G^m$, where $G$ is the homocyclic abelian group of
order $q^e$ which is the direct sum of $d$ cyclic groups of order $p^e$ (so
that its Frattini quotient is the additive group of the field of order~$q$). 
We take the $R$-modules generated by these vectors; each is (additively) a
subgroup isomorphic to $G$. Thus, we have a set of $q+1$ subgroups, any $m$ of
which generate their direct sum, and so an example of a regular set of MODS
where all the groups 
are isomorphic to $(C_{p^e})^d$.


\begin{thebibliography}{99}

\bibitem{albert}
A. A. Albert,
Quasigroups, I,
\textit{Trans. Amer. Math. Soc.} \textbf{54} (1943), 507--519.

\bibitem{bcps}
R. A. Bailey, P. J. Cameron, C. E. Praeger and C. Schneider,
The geometry of diagonal groups,
arXiv \texttt{2007.10726}

\bibitem{bl}
S. Ball and M. Lavrouw,
Arcs in finite projective spaces,
\textit{Europ. Math. Soc. Surveys} \textbf{6} (2020), 133--172.

\bibitem{bbt}
A. Blokhuis, A. A. Bruen and J. A. Thas,
On M.D.S. codes, arcs in $\PG(n,q)$ with $q$ even, and a solution of three
fundamental problems of B. Segre,
\textit{Invent. Math.} \textbf{92} (1988), 441--459.

\bibitem{bccsz}
J. N. Bray, Q. Cai, P. J. Cameron, P. Spiga and H. Zhang,
The Hall--Paige conjecture, and synchronization for affine and diagonal groups,
\textit{J. Algebra} \textbf{545} (2020), 27--42.

\bibitem{cohen}
Henri Cohen, \textit{Number Theory, Volume I: Tools and Diophantine Equations},
Graduate Texts in Mathematics \textbf{239}, Springer, New York, 2007.

\bibitem{dk}
J. D\'enes and A. D. Keedwell,
\textit{Latin Squares and their Applications},
Akad\'emai Kaid\'o, Budapest, and English Universities Press, London, 1974.

\bibitem{ds}
Steven T. Dougherty and Theresa A. Szczepanski,
Latin $k$-hypercubes,
\textit{Australas. J. Combinatorics} \textbf{40} (2008), 145--160.

\bibitem{EW}
Judith Egan and Ian M. Wanless,
Enumeration of MOLS of small order,
\textit{Mathematics of Computation},
\textbf{85} (2016), 799--824.

\bibitem{evans}
Anthony B. Evans,
The admissibility of sporadic simple groups,
\textit{J. Algebra} \textbf{321} (2009), no. 1, 105--116.

\bibitem{FDE}
R. A. Fisher,
\textit{The Design of Experiments},
Oliver \& Boyd, Edinburgh, 1935.
  
  
\bibitem{Fapol}
R. A. Fisher,
Completely orthogonal $9\times 9$ Latin squares. A correction,
\textit{Annals of Eugenics} \textbf{11} (1942), 402--403.
  
\bibitem{FY}
R. A. Fisher and F. Yates,
\textit{Statistical Tables for Biological, Agricultural and Medical
Research},
Oliver \& Boyd, Edinburgh, 1938.

\bibitem{fhw}
Nevena Franceti\'c, Sarada Herke, and Ian M. Wanless,
Parity of sets of mutually orthogonal Latin squares,
\textit{J. Combinatorial Theory} (A) \textbf{155} (2018), 67--99.
    
\bibitem{frolov}
M. Frolov,
Recherches sur les permutations carr\'ees,
\textit{J. Math. Sp\'ec.} (3) \textbf{4} (1890), 8--11.

\bibitem{hp}
Marshall Hall Jr. and L. J. Paige,
Complete mappings of finite groups,
\textit{Pacific J. Math.} \textbf{5} (1955), 541--549.

\bibitem{HSS}
A.~S.~Hedayat, N.~J.~A.~Sloane and John Stufken,
\textit{Orthogonal Arrays: Theory and Applications},
Springer Series in Statistics, Springer-Verlag, New York, 1999.  

\bibitem{mw}
Brendan D. McKay and Ian M. Wanless,
A census of small Latin hypercubes.
\textit{Siam J.\ Discrete Math.} \textbf{22} (2008), 719--736.

\bibitem{mullen}
Gary L. Mullen,
Orthogonal hypercubes and related designs,
\textit{J. Statist. Planning Inference} \textbf{73} (1998), 177--188.

\bibitem{OP1}
P.~J.~Owens and D.~A.~Preece,
Complete sets of pairwise orthogonal Latin squares of order~$9$, 
\textit{J. Combin. Math. Combin. Comput.} \textbf{18} (1995), 83--96. 

\bibitem{OP2}
P.~J.~Owens and D.~A.~Preece,
Aspects of complete sets of $9 \times 9$ pairwise orthogonal Latin squares, 
\textit{Discrete Math.} \textbf{167/168} (1997), 519--525.

\bibitem{passman}
D. S. Passman,
\textit{Permutation Groups},
Dover Publications
(revised republication of the work orginally published in 1968 by the
W.~A.~Benjamin Company),
New York, 2012.

\bibitem{potapov}
Vladimir N. Potapov,
Constructions of pairs of orthogonal Latin cubes,
\textit{J. Combinatorial Designs} \textbf{28} (2020), 604--613.

\bibitem{segre}
B. Segre, 
Curve razionali normali e $k$-archi negli spazi finiti,
\textit{Ann. Mat. Pura Appl.} (4) \textbf{39} (1955), 357--379.

\bibitem{trenkler}
Mari\'an Trenkler,
On orthogonal Latin $p$-dimensional cubes,
\textit{Czechosolvak Math. J.} \textbf{55} (2005), 725--728.

\bibitem{wilcox}
Stewart Wilcox,
Reduction of the Hall--Paige conjecture to sporadic simple groups,
\textit{J. Algebra} \textbf{321} (2009), no. 5, 1407--1428.

\end{thebibliography}
\end{document}